\documentclass[a4paper,12pt,oneside,dvipsnames,reqno]{amsart} 

\usepackage{amsmath,amsfonts,amsthm,amssymb,mathtools}

\usepackage{xcolor}
\usepackage{graphicx}
\usepackage{cite}

\usepackage{caption}
\usepackage{subcaption}

\usepackage{bbm}
\usepackage{tikz-cd}

\usepackage{hyperref}
\hypersetup{colorlinks=true, citecolor=PineGreen, linkcolor=RoyalBlue, linktoc=page,urlcolor=RoyalBlue, 
}

\usepackage{geometry}

\theoremstyle{plain}
\newtheorem{theorem}{Theorem}[section] 
\newtheorem{lemma}[theorem]{Lemma}
\newtheorem{prop}[theorem]{Proposition}
\newtheorem{cor}[theorem]{Corollary}

\newtheorem{conj}{Conjecture}[section]

\theoremstyle{remark}
\newtheorem{rem}[theorem]{Remark}

\theoremstyle{definition}

\newtheorem{example}[theorem]{Example}

\newcommand{\Hb}{\mathbb{H}}

\newcommand{\GL}{\operatorname{GL}}

\newcommand{\n}{\operatorname{n}}

\renewcommand{\Re}{\operatorname{Re}}
\renewcommand{\Im}{\operatorname{Im}}

\newcommand{\diag}{\operatorname{diag}}

\newcommand{\arsinh}{\operatorname{arsinh}}

\title{Spectral fourth moments of Hecke--Maa\ss{} cusp forms}

\author{Edgar Assing}
\address{Mathematical Institute of the University of Bonn, Endenicher Allee 60, D–53115 Bonn, Germany}
\email{\href{mailto:assing@math.uni-bonn.de}{assing@math.uni-bonn.de}}

\author{Peter Humphries}
\address{Department of Mathematics, University of Virginia, Charlottesville, VA 22904, USA}
\email{\href{mailto:pclhumphries@gmail.com}{pclhumphries@gmail.com}}
\urladdr{\href{https://sites.google.com/view/peterhumphries/}{https://sites.google.com/view/peterhumphries/}}
\date{\today}

\begin{document}

\begin{abstract}
In this note, we establish essentially optimal bounds for certain spectral moments of automorphic forms for $\mathrm{GL}(2)$. More precisely, we consider the family of Hecke--Maa\ss{} cusp forms with spectral parameter in a dyadic interval and study the fourth moment of these forms evaluated at a Heegner point. We additionally present applications of our main result to the shifted convolution problem involving the sum of two squares function $r(n)$ as well as to pointwise Diophantine exponents.
\end{abstract}

\maketitle

\section{Introduction}

\subsection{The main result}

Let $q\in \mathbb{N}$ be fixed and let $\Gamma_0(q)\subseteq \mathrm{SL}_2(\mathbb{Z})$ denote the Hecke congruence subgroup. We recall the spectral decomposition of $L^2(\Gamma_0(q)\backslash \mathbb{H})$ with respect to the Laplace--Beltrami operator $\Delta = -y^2(\frac{\partial^2}{\partial x^2}+\frac{\partial^2}{\partial y^2})$; see, for example, \cite{Iw}. Our main focus lies on $L_{\mathrm{cusp}}^2(\Gamma_0(q)\backslash \mathbb{H})$, the cuspidal part of this space. We write $\mathcal{B}(q)$ for an orthonormal basis of $L_{\mathrm{cusp}}^2(\Gamma_0(q)\backslash \mathbb{H})$ that consists of eigenfunctions of $\Delta$. Given $\varphi\in \mathcal{B}(q)$, we write $\lambda_{\varphi} = \frac{1}{4}+t_{\varphi}^2$ for the corresponding $\Delta$ eigenvalue, where $t_{\varphi} \in [0,\infty) \cup i(0,\frac{1}{2})$ denotes the associated spectral parameter. We also assume that the elements $\varphi\in \mathcal{B}(q)$ are eigenfunctions of all Hecke operators $T_m$ with $(m,q)=1$ and we write $\lambda_{\varphi}(m)$ for the corresponding Hecke eigenvalues (i.e.\ $T_m\varphi = \lambda_{\varphi}(m)\cdot \varphi$). Elements of such a basis $\mathcal{B}(q)$ are called Hecke--Maa\ss{} cusp forms.

We are interested in estimating the spectral fourth moment
\begin{equation}
	\sum_{\substack{\varphi\in \mathcal{B}(q)\\ T\leq t_{\varphi} \leq  2T}} \vert \varphi(z_0)\vert^4 \label{eq:moment}
\end{equation}
as $T\to \infty$ with both $q \in \mathbb{N}$ and $z_0\in \Gamma_0(q)\backslash \mathbb{H}$ fixed. Recall that by the local Weyl law,
\[\sum_{\substack{\varphi\in \mathcal{B}(q)\\ T\leq t_{\varphi} \leq  2T}} \vert \varphi(z_0)\vert^2 \ll_{q,z_0} T^2.\]
See \cite[(7.10)]{Iw} for a precise statement that applies to the setting at hand. Furthermore, we have the local bound
\[\vert \varphi(z_0)\vert \ll_{q,z_0} (1+|t_{\varphi}|)^{\frac{1}{2}}\]
as stated in \cite[(13.9)]{Iw}. Thus, the trivial bound for \eqref{eq:moment} is
\[\sum_{\substack{\varphi\in \mathcal{B}(q)\\ T\leq t_{\varphi} \leq  2T}} \vert \varphi(z_0)\vert^4 \ll_{q,z_0} T^3.\]
This is far from the best possible. Indeed, a (by now) folklore conjecture, stated in \cite[(13.2)]{Iw}, predicts that for all $\epsilon>0$,
\[\vert \varphi(z_0)\vert \ll_{q,z_0,\epsilon} (1+|t_{\varphi}|)^{\epsilon}\]
for all $\varphi\in \mathcal{B}(q)$ and. In view of this together with the Weyl law
\[\sharp \{ \varphi\in \mathcal{B}(q)\colon T\leq t_{\varphi} \leq 2T\} \asymp_{q} T^2\]
as stated in \cite[(11.5)]{Iw} the following conjecture of Chamizo can be seen as a \textit{Lindel\"of on average} bound for \eqref{eq:moment}.

\begin{conj}[{\cite[Conjecture~1.1]{Ch}}]\label{conj:1.1}
For $q\in \mathbb{N}$, $z_0\in \Gamma_0(q)\backslash \mathbb{H}$, and $T\geq 1$, we have that
\begin{equation}
	\sum_{\substack{\varphi\in \mathcal{B}(q)\\ T\leq t_{\varphi} \leq  2T}} \vert \varphi(z_0)\vert^4 \ll_{q,x_o,\epsilon} T^{2+\epsilon}.\label{eq:4th}
\end{equation}
\end{conj} 

\begin{rem}
Note that our focus here is purely on the spectral aspect in the sense that we grow the family $\{ \varphi\in \mathcal{B}(q)\colon T\leq t_{\varphi} \leq 2T\}$ by taking $T\to \infty$ while keeping $q$ fixed. Reversing the roles of $T$ and $q$ leads to the so-called level aspect. In this setting, a related fourth moment has been successfully studied in the series of papers \cite{S1, S2} using theta functions. Unfortunately, this approach has so far not seen success in the spectral aspect. Interestingly, Steiner has reduced Conjecture~\ref{conj:1.1} (and more) to an intricate matrix counting problem \cite[Theorem~3]{S3}. To the best of our knowledge, this counting problem has resisted all attempts to prove it.	
\end{rem}

Since currently we are unable to resolve Conjecture~\ref{conj:1.1} for general $z_0\in \Gamma_0(q)\backslash\mathbb{H}$, we restrict our attention to so-called Heegner points associated to negative fundamental discriminants $D < 0$; see Section~\ref{sec:Heegner} below for a precise definition. Our main result is the following.

\begin{theorem}\label{th:4th_heegner}
Let $q\in \mathbb{N}$ and $T\ge 1$. For a Heegner point $z_0\in\Gamma_0(q)\backslash \mathbb{H}$ of discriminant $D$ and level $q$, we have that for all $\epsilon > 0$,
\[\sum_{\substack{\varphi\in \mathcal{B}(q)\\ T\leq t_{\varphi} \leq  2T}} \vert \varphi(z_0)\vert^4 \ll_{q,D,\epsilon} T^{2+\epsilon}.\]
\end{theorem}

\subsection{Application to correlated sums of \texorpdfstring{$r(n)$}{r(n)}}

For $n\in \mathbb{N}$, we let $r(n)$ denote the number of representations of $n$ as the sum of two squares. For $N,m \in \mathbb{N}$, we consider the correlation
\[S(N,m) = \sum_{n\leq N} r(n)r(n+m).\]
It is well known that we can write 
\[S(N,m) = \frac{8c_m}{m}\cdot N+E(N,m) \quad \text{for }c_m = \sum_{d\mid m}(-1)^{m+d}d,\]
where the error term satisfies $E(N,m)=o_m(N)$. A problem of interest is improving this error estimate and making it uniform in the shift $m$. This was considered by Estermann in \cite{Es}, who showed that
\[E(N,m)\ll_{m,\epsilon} N^{\frac{11}{12}+\epsilon}.\]
A milestone was achieved by Iwaniec \cite[Theorem~12.5]{Iw}, who improved the error term estimate to
\[E(N,m) \ll_{\epsilon} m^{\frac{1}{3}+\epsilon}N^{\frac{2}{3}} \quad \text{for }1\leq m\leq N \text{ odd.}\]
Chamizo removed the assumption that $m$ is odd and showed in \cite{Ch} that
\begin{equation}
    E(N,m) \ll_{\epsilon}N^{\frac{2}{3}}+m^{\frac{1}{3}+\epsilon}N^{\frac{1}{3}}
\label{eq:ram_pet}
\end{equation}
under the assumption of the Ramanujan--Petersson conjecture.\footnote{We have implicitly assumed that the Hecke operators $T_m$ are normalised such that $\lambda_{\varphi}(m)\ll_{\epsilon} m^{\frac{1}{2}+\epsilon}$ is the trivial bound. The Ramanujan--Petersson conjecture then predicts that $\lambda_{\varphi}(m)\ll_{\epsilon}m^{\epsilon}$.} The currently best unconditional results can be found in \cite[Corollary~1.5]{Ch2} and read
\begin{equation}
E(N,m) \ll_{\epsilon} m^{\epsilon}\cdot \begin{dcases*}
		N^{\frac{2}{3}} & if $m^{\frac{3}{2} + 3\theta}\leq N$, \\
		m^{\frac{1 + 2\theta}{4}}N^{\frac{1}{2}} & if $m^{\frac{3}{2} - \theta}\leq N <  m^{\frac{3}{2} + 3\theta}$,\\
		m^{\frac{2\theta}{3}}N^{\frac{2}{3}} & if $m^{\frac{46\theta}{5}}\leq N < m^{\frac{3}{2} - \theta}$,\\
		N^{\frac{17}{23}} & if $m\leq N < m^{\frac{46\theta}{5}}$, \\
		(mN)^{\frac{17}{46}} & if $m^{\frac{61 + 92\theta}{77}} \leq  N < m$, \\
		m^{\frac{13 + 4\theta}{28}}N^{\frac{1}{4}} & if $m^{\frac{11 - 44\theta}{7}} \leq N < m^{\frac{61 + 92\theta}{77}}$, \\
		m^{\frac{1 + 2\theta}{3}}N^{\frac{1}{3}} & if $N <  m^{\frac{11 - 44\theta}{7}}$,
	\end{dcases*} \label{eq:uncond_chami}
\end{equation}
where $\theta\geq 0$ is the best possible bound towards the Ramanujan--Petersson conjecture.\footnote{More precisely, $\theta\geq 0$ is the smallest exponent for which $\lambda_{\varphi}(m) \ll_{\epsilon} m^{\theta+\epsilon}$ holds unconditionally for all $m\in \mathbb{N}$. The current record, due to Kim and Sarnak \cite[Appendix~2]{Kim}, is $\theta=\frac{7}{64}$.} We refer to Figure~\ref{fig} for a visualization of this complicated looking bound.

As a consequence of our main results, we can make \cite[Theorem~3.1]{Ch} (see also \cite[Theorem~4.1]{Ch2}) unconditional, achieving the following estimate for the error term $E(N,m)$.

\begin{theorem}\label{main_th1}
For all $\epsilon >0$, we have
\begin{equation}
E(N,m) \ll_{\epsilon} m^{\epsilon}\cdot \begin{dcases*}
		N^{\frac{2}{3}} & if $m^{\frac{3}{4} + 3\theta}\leq N$, \\
		m^{\frac{1 + 4\theta}{8}}N^{\frac{1}{2}} & if $m \leq N <  m^{\frac{3}{4} + 3\theta}$,\\
		m^{\frac{3 + 4\theta}{8}}N^{\frac{1}{4}} & if $m^{\frac{1}{2} + 2\theta}\leq N < m$,\\
		  m^{\frac{1}{4}}N^{\frac{1}{2}} & if $m^{\frac{1}{2}} \leq N < m^{\frac{1}{2} + 2\theta}$, \\
		  m^{\frac{1}{3}}N^{\frac{1}{3}} & if $N <  m^{\frac{1}{2}}$.\
	\end{dcases*} \label{eq:cond_on_fourth}
\end{equation}
Furthermore, we have the mean square estimate
\[\sum_{M<m\leq 2M} \vert E(N,m)\vert^2 \ll_{\epsilon} N^{\epsilon}(N^{\frac{4}{2}}M^{\frac{1}{3}} + M^4),\]
for $1 < M^2 < N$.
\end{theorem}
\begin{proof}
The pointwise bound was stated in \cite[Theorem~3.1]{Ch} and \cite[Theorem~4.1]{Ch2} conditional on Conjecture~\ref{conj:1.1}. Similarly, the mean square estimate appears in  \cite[Corollary~4.2]{Ch} also conditional on Conjecture~\ref{conj:1.1}. Inspecting the proofs shows that they do not require Conjecture~\ref{conj:1.1} in its full strength. Indeed, it suffices to prove \eqref{eq:4th} in two situations: first, for $q=1$ and $z_0=i$, and second, for $q=2$ and $z_0=\frac{-1+i}{2}$. Both of these points $z_0$ are Heegner points of discriminant $D=-4$, so that the desired results follow from Theorem~\ref{main_th1}.
\end{proof}

\begin{figure}
	\centering
	\includegraphics[width=0.95\linewidth]{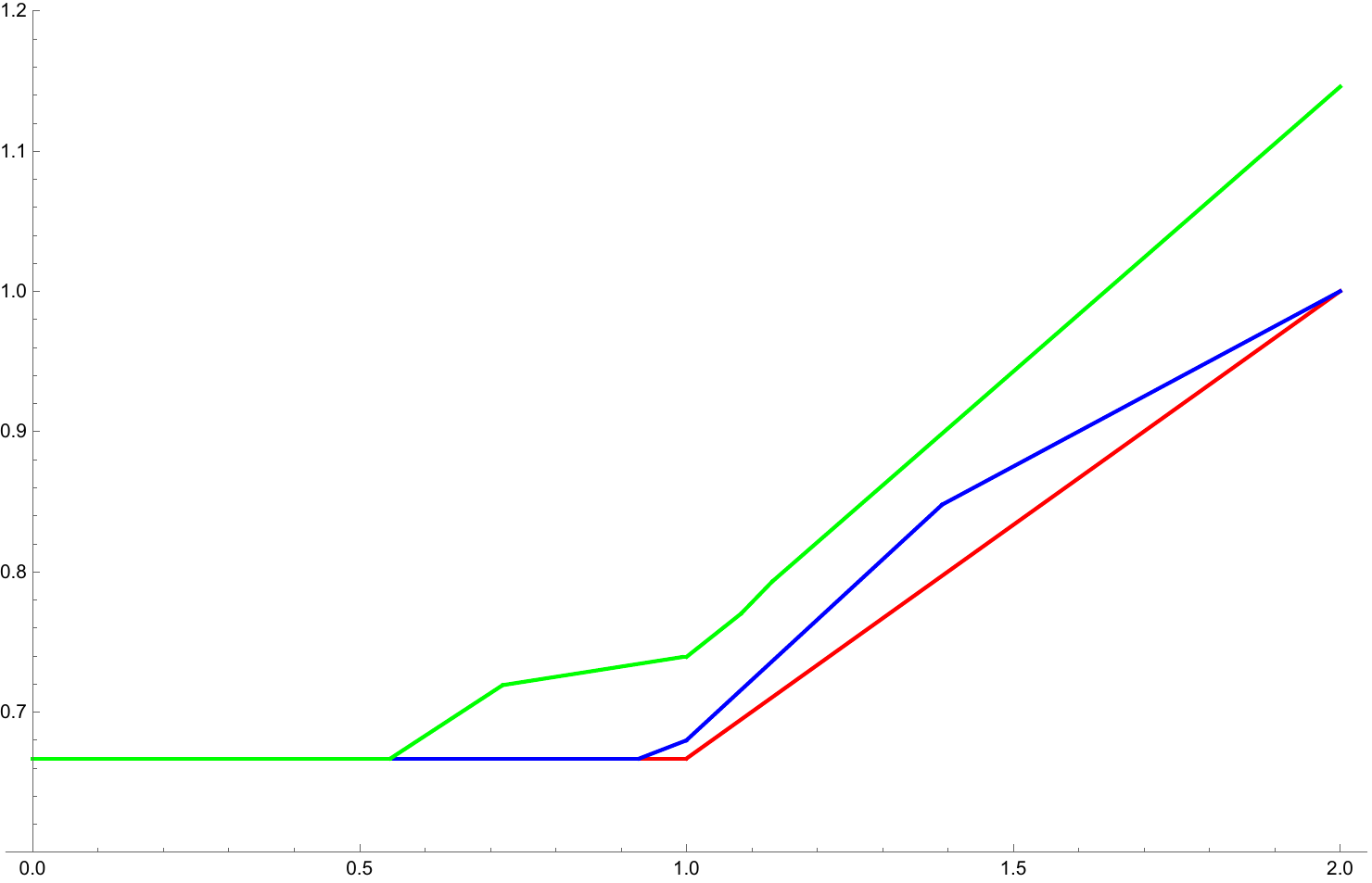}
	\caption{We compare different exponents $\beta(\alpha)$ arising in the estimate $E(N,N^{\alpha})\ll_{\epsilon} N^{\beta(\alpha)+\epsilon}$, where $\alpha\in [0,2]$. The exponent in red is conditional on the Ramanujan--Petersson conjecture and originates from \eqref{eq:ram_pet}. We display the exponents arising from \eqref{eq:cond_on_fourth} in blue, which we establish unconditionally in Theorem~\ref{main_th1}. Finally, we show the exponents from \eqref{eq:uncond_chami} in green. In the latter cases, we have used the current record $\theta=\frac{7}{64}$.}
	\label{fig}
\end{figure}

\begin{rem}
Determining the asymptotic behaviour of $S(N,m)$ is a prototypical example of a shifted convolution problem, where one studies the shifted convolution sum $\sum_{n \leq N} a(n) b(n + m)$ for some chosen arithmetic functions $a,b : \mathbb{N} \to \mathbb{C}$. Of particular interest is the setting where both $a$ and $b$ arise as the Hecke eigenvalues of automorphic forms. For the shifted convolution sum $S(N,m)$, both $a$ and $b$ are the Hecke eigenvalues $r(n)$ of an Eisenstein series of weight $1$, level $4$, and nebentypus $\chi_{-4}$, the quadratic character modulo $4$. A related problem, which naturally shows up in connection with the fourth moment of the Riemann zeta function, is the binary additive divisor sum
\[
D(N,m) = \sum_{n\leq N}\tau(n) \tau(n+m)\]
involving the divisor function $\tau(\n)$. In this setting, the arithmetic functions $a$ and $b$ are both the divisor function $\tau(n)$, which are the Hecke eigenvalues of the derivative at $s = \frac{1}{2}$ of the Eisenstein series $E(z,s)$ of level $1$. This problem has been extensively studied, most notably by Motohashi \cite{Mo}. The currently best known bounds for the error term in this problem that are uniform in the shift $m$ coincide with the bounds established in Theorem~\ref{main_th1}. We refer to \cite{BF} and the references within for more discussion concerning the binary additive divisor problem.
\end{rem}

\begin{rem}
Analogous problems can be studied in the function field setting; see \cite[Corollary~1.4]{FLMS}.
\end{rem}

\begin{rem}
It is a well known feature that the error estimate for shifted convolution problems can be improved considerably when the sharp cut-off $n\leq N$ is replaced by a smooth one. In our context, a very nice result was recently obtained in \cite[Theorem~7.1]{Hulse}: they allow for exponential smoothing and obtain
\[\sum_{n = 1}^{\infty} r(n)r(n+m) e^{-n/N} = C_m\cdot N + O_{\epsilon}(N^{\frac{1}{2}+\epsilon}e^{\frac{m}{N}}m^{\theta}).\]
\end{rem}

\subsection{Application to pointwise Diophantine exponents}

Let $\rho : \mathbb{H} \times \mathbb{H} \to [0,\infty)$ denote the $\mathrm{SL}_2(\mathbb{R})$-invariant Riemannian metric on $\mathbb{H}$ given by
\[\rho(z,w) = 2 \arsinh \frac{|z - w|}{2\sqrt{\Im(z)\Im(w)}} = \log \frac{|z - \overline{w}| + |z - w|}{|z - \overline{w}| - |z - w|}.\]
Let $\mu$ denote the associated Riemannian volume form on $\mathbb{H}$, namely $d\mu(z) = y^{-2} \, dx \, dy$ for $z = x + iy \in \mathbb{H}$. Fix a prime $p$, and for $\gamma \in \mathrm{SL}_2(\mathbb{Z}[\frac{1}{p}])$, let
\[\mathrm{ht}(\gamma) = \min\left\{k \in \mathbb{N} : p^k\gamma \in \mathrm{Mat}_2(\mathbb{Z})\right\}.\]
We recall the following definitions from \cite{JK}:
\begin{itemize}
\item The pointwise Diophantine exponent $\kappa(z,w)$ of $(z,w) \in \mathbb{H} \times \mathbb{H}$ is
\begin{multline*}
\kappa(z,w) = \inf\left\{\zeta : \exists \varepsilon_0 > 0 \text{ s.t. } \forall \varepsilon \in (0,\varepsilon_0) \ \exists \gamma \in \mathrm{SL}_2\left(\mathbb{Z}\left[\frac{1}{p}\right]\right) \text{ s.t.} \right. \\
\left. \vphantom{\mathrm{SL}_2\left(\mathbb{Z}\left[\frac{1}{p}\right]\right)} \rho(\gamma^{-1} z,w) \leq \varepsilon, \ \mathrm{ht}(\gamma) \leq \zeta \log_p \frac{1}{\varepsilon}\right\}.
\end{multline*}
\item The pointwise Diophantine exponent $\kappa(z)$ of $z \in \mathbb{H}$ is
\[\kappa(z) = \inf\left\{\tau : \kappa(z,w) \leq \tau \text{ for $\mu$-a.e.\ } w \in \mathbb{H}\right\}.\]
\item The Diophantine exponent $\kappa$ of $\mathbb{H}$ is
\[\kappa = \inf\left\{\tau : \kappa(z) \leq \tau \text{ for $\mu$-a.e.\ } z \in \mathbb{H}\right\}.\]
\end{itemize}

In the latter two cases, we have the trivial lower bounds $\kappa(z),\kappa \geq 1$. In \cite{JK}, Jana and Kamber raise the question of whether these inequalities are sharp, so that $\kappa(z) = \kappa = 1$ for all $z \in \mathbb{H}$. Recently, Jana and Kamber established in \cite[Theorem~2]{JK} the latter, namely that $\kappa=1$. For the former, it was shown by Ghosh, Gorodnik, and Nevo \cite{GGN} that for all $z_0\in \mathbb{H}$, the upper bound $\kappa(z_0)\leq \frac{32}{25}$ holds unconditionally and the equality $\kappa(z_0) = 1$ holds under the assumption of the Ramanujan--Petersson conjecture; see \cite[Theorem~1]{JK}. As a corollary of our main result, we obtain the following partial result.

\begin{cor}
\label{cor:diophantine}
For a Heegner point $z_0\in\mathrm{SL}_2(\mathbb{Z})\backslash \mathbb{H})$, we have that $\kappa(z_0)=1$.
\end{cor}
\begin{proof}
Following the argument from \cite[Section~8]{JK}, one reduces proving that $\kappa(z_0)=1$ to a local version of the density hypothesis displayed in \eqref{eq:enhan_density} below. The starting point is an application of \cite[Lemma~5.11 and Proposition~6.1]{JK} with $n=2$, $\beta=1$, $\epsilon=\frac{1}{T}$ sufficiently small, and $k = \lfloor\log_p(T)\rfloor$. This reduces the estimate $\kappa(z_0)\leq 1$ to the bound
\begin{equation}
    \sum_{\substack{\varphi\in \mathcal{B}(1) \\ t_{\varphi}\leq T^{1+\epsilon}}}\vert \lambda_{\varphi}(p^{2k})\vert^2\cdot \vert \varphi(z_0)\vert^2 + \int_{\vert t\vert\leq T^{1+\epsilon}} \left\vert E\left(z_0,\frac{1}{2}+it\right)\right\vert^2 \, dt \ll_{p,z_0,\epsilon} p^{2k}T^{\epsilon}.\label{eq:to_show_for_optexp}
\end{equation}
Here we note that, for $\textrm{PGL}_2$, the Eisenstein spectrum is everywhere tempered and hence poses no problem. Indeed, we can estimate the continuous contribution in \eqref{eq:to_show_for_optexp} using well known bounds for Eisenstein series; see, for example, \cite[Proposition~3]{purity}. The remaining task is to bound the cuspidal contribution in \eqref{eq:to_show_for_optexp}. This can easily be done using the following inequality:
\begin{equation}
	\sum_{\substack{\varphi\in \mathcal{B}(1) \\ T\leq t_{\varphi}\leq 2T}} \vert \lambda_{\varphi}(m)\vert^2\cdot \vert \varphi(z_0)\vert^2 \ll_{z_0,\epsilon}  (mT)^{\epsilon} (T^2+m). \label{eq:enhan_density}
\end{equation}

We now establish \eqref{eq:enhan_density} for Heegner points $z_0$. We begin by applying the Cauchy--Schwarz inequality and Theorem~\ref{th:4th_heegner} in order to obtain
\[\sum_{\substack{\varphi\in \mathcal{B}(1) \\ T\leq t_{\varphi}\leq 2T}} \vert \lambda_{\varphi}(m)\vert^2\cdot \vert \varphi(z_0)\vert^2 \ll_{z_0,\epsilon}  T^{1+\epsilon}\cdot \left(\sum_{\substack{\varphi\in \mathcal{B}(1) \\ T\leq t_{\varphi}\leq 2T}}\vert \lambda_{\varphi}(m)\vert^4\right)^{\frac{1}{2}}.\]
To handle the fourth moment of the Hecke eigenvalues, we recall the Hecke relation
\[\lambda_{\varphi}(m)\lambda_{\varphi}(n) = \sum_{d\mid (m,n)} \lambda_{\varphi}\left(\frac{nm}{d^2}\right).\]
We can thus estimate
\[\sum_{\substack{\varphi\in \mathcal{B}(1) \\ T\leq t_{\varphi}\leq 2T}}\vert \lambda_{\varphi}(m)\vert^4 \ll_{\epsilon} m^{\epsilon} \max_{\ell \mid m^2} \sum_{\substack{\varphi\in \mathcal{B}(1) \\ T\leq t_{\varphi}\leq 2T}}\vert \lambda_{\varphi}(\ell)\vert^2.\]
The remaining sum can be estimated using\footnote{In the context of the density hypothesis, versions of \eqref{eq:dens_stro} appear, for example, in \cite{BBR} and \cite{Hu}. Let us stress that the estimate stated in \eqref{eq:dens_stro} is actually stronger than the original density hypothesis as proven in \cite{Sa} using the Selberg trace formula.}
\begin{equation}
	\sum_{\substack{\varphi\in \mathcal{B}(q) \\ T\leq t_{\varphi}\leq 2T}} \vert \lambda_{\varphi}(\ell)\vert^2 \ll_{q,\epsilon}  (\ell T)^{\epsilon} (T^2+\sqrt{\ell}). \label{eq:dens_stro}
\end{equation}
This can be derived from the Kuznetsov formula and is stated, for example, in the proof of \cite[Lemma~3.3]{Ch}; see also \cite[Lemma~2.4]{Ch2}. Putting everything together leads to
\[\sum_{\substack{\varphi\in \mathcal{B}(1) \\ T\leq t_{\varphi}\leq 2T}} \vert \lambda_{\varphi}(m)\vert^2\cdot \vert \varphi(z_0)\vert^2 \ll_{z_0,\epsilon}  m^{\epsilon}T^{1+\epsilon}\cdot \left(T^2+m\right)^{\frac{1}{2}}.\]
This evidently implies \eqref{eq:enhan_density} and completes the proof.
\end{proof}

\begin{rem}
It is already pointed out in \cite[Remark~8.3]{JK}, that \eqref{eq:enhan_density} is the correct variation of the density hypothesis required to establish optimal bounds for the pointwise Diophantine exponents $\kappa(z_0)$. In the same remark, Jana and Kamber mention the possibility of establishing \eqref{eq:enhan_density} for \textit{some $z_0$ using different methods.} To the best of our knowledge, this proof has never appeared in print. In private communication, we have been told that the envisioned argument is dynamical in nature and applies to $z_0$ with certain Diophantine properties, but not necessarily to Heegner points.
\end{rem}

\begin{rem}
Our discussion of Diophantine exponents focuses exclusively on $\mathbb{H}$. However, the definitions apply in great generality and it is an interesting problem to study versions of $\kappa$ and $\kappa(z)$ in different settings. In general, the pointwise Diophantine exponent $\kappa(z)$ can be very sensitive to the point $z$. A concrete example of such phenomena was worked out in \cite[Corollary~1.7]{Sard}, where the pointwise Diophantine exponent is considered for the spheres and turns out to be much larger than the Diophantine exponent $\kappa$ in some cases; cf.~\cite[Section~2.1]{JK}. This is in sharp contrast to what we observe here in the case of $\mathbb{H}$.
\end{rem}

The proof of Corollary~\ref{cor:diophantine} reduces to proving the inequality \eqref{eq:enhan_density}. The proof that we gave above of this inequality uses as an input Theorem~\ref{th:4th_heegner}. The inequality \eqref{eq:enhan_density} also follows from the following result.

\begin{theorem}
\label{thm:twistedsecondmoment}
For a Heegner point $z_0\in \mathrm{SL}_2(\mathbb{Z})\backslash \mathbb{H}$, we have that
\begin{multline*}
\sum_{\varphi\in \mathcal{B}(1)} \lambda_{\varphi}(\ell) \cdot \vert \varphi(z_0)\vert^2 e^{-\frac{t_{\varphi}^2}{T^2}} + \frac{1}{4\pi} \int_{-\infty}^{\infty} \lambda(\ell,t) \cdot \left\vert E\left(z_0,\frac{1}{2} + it\right)\right\vert^2 e^{-\frac{t^2}{T^2}} \, dt    \\
\ll_{z_0,\epsilon} (\ell T)^{\epsilon} \left(\frac{T^2}{\sqrt{\ell}} + \sqrt{\ell}\right).
\end{multline*}
\end{theorem}

Here $\lambda(\ell,t) = \sum_{ab = \ell} a^{it} b^{-it}$ denotes the $\ell$-th Hecke eigenvalue of the real analytic Eisenstein series $E(z,\frac{1}{2} + it)$.

\begin{proof}[Second proof of \eqref{eq:enhan_density}]
Via positivity, the Hecke relations
\[\lambda_{\varphi}(m) = \sum_{\ell \mid m} \lambda_{\varphi}(\ell^2), \qquad \lambda(m,t) = \sum_{\ell \mid m} \lambda(\ell^2,t),\]
and Theorem~\ref{thm:twistedsecondmoment},
\begin{align*}
& \sum_{\substack{\varphi\in \mathcal{B}(1) \\ T \leq t_{\varphi} \leq 2T}} \vert\lambda_{\varphi}(m)\vert^2 \cdot \vert \varphi(z_0)\vert^2	\\
& \leq \sum_{\varphi\in \mathcal{B}(1)} \vert\lambda_{\varphi}(m)\vert^2 \cdot \vert \varphi(z_0)\vert^2 e^{-\frac{t_{\varphi}^2}{T^2}} + \frac{1}{4\pi} \int_{-\infty}^{\infty} \vert\lambda(m,t)\vert^2 \cdot \left\vert E\left(z_0,\frac{1}{2} + it\right)\right\vert^2 e^{-\frac{t^2}{T^2}} \, dt	\\
& = \sum_{\ell \mid m} \left(\sum_{\varphi\in \mathcal{B}(1)} \lambda_{\varphi}(\ell^2) \cdot \vert \varphi(z_0)\vert^2 e^{-\frac{t_{\varphi}^2}{T^2}} + \frac{1}{4\pi} \int_{-\infty}^{\infty} \lambda(\ell^2,t) \cdot \left\vert E\left(z_0,\frac{1}{2} + it\right)\right\vert^2 e^{-\frac{t^2}{T^2}} \, dt\right) \hspace{-1cm}	\\
& \ll_{z_0,\epsilon} (mT)^{\epsilon}(T^2 + m).\qedhere
\end{align*}
\end{proof}

In Section \ref{sec:pretrace}, we prove Theorem~\ref{thm:twistedsecondmoment} by more classical means, namely via the pretrace formula combined with a counting argument. This circumvents the usage of Waldspurger's formula to prove \eqref{eq:enhan_density} (cf. Remark \ref{rem:HK}). Our proof is inspired by the proof of \cite[Lemma~3.1]{Ch2}, where this result is proven in the special case $z_0 = i$.


\textbf{Acknowledgement:} The first author would like to thank V.~Blomer for his support and encouragement as well as A.~Pascadi for helpful conversations concerning (relevant) shifted convolution problems. Furthermore, both authors are grateful to F.~Chamizo for bringing them together as well as to S.~Jana for useful comments on an earlier version of the paper.  The first author is supported by the ERC Advanced Grant 101054336, the Germany Excellence Strategy grant EXC-2047/1-390685813, and also partially by the Deutsche Forschungsgemeinschaft (DFG, German Research Foundation) – Project-ID 491392403 – TRR 358. The second author is supported by the National Science Foundation (grant DMS-2302079) and by the Simons Foundation (award 965056).

\section{Preliminaries}\label{sec:prelim}

\subsection{Hecke--Maa\ss{} forms}

For $q\in \mathbb{N}$, the Hecke congruence subgroup of $\mathrm{SL}_2(\mathbb{Z})$ is defined by
\[\Gamma_0(q) = \left\{\left(\begin{matrix} a & b \\ qc&d\end{matrix}\right)\colon a,b,c,d\in \mathbb{Z} \text{ and }ad-qbc=1  \right\} \subseteq \mathrm{SL}_2(\mathbb{Z}).\]
We write $\Gamma$ for $\Gamma_0(1) = \mathrm{SL}_2(\mathbb{Z})$. The group $\Gamma_0(q)$ acts on the upper half plane $\mathbb{H}$ via M\"obius transformations and we can form the quotient space $\Gamma_0(q)\backslash \mathbb{H}$. This is an orbifold, which we equip with the probability Haar measure.

Recall that $\mathcal{B}(q)$ denotes an orthonormal basis of the space $L_{\mathrm{cusp}}^2(\Gamma_0(q)\backslash \mathbb{H})$ consisting of Hecke--Maa\ss{} cusp forms.
These are eigenfunctions of $\Delta$ with eigenvalue $\frac{1}{4}+t_{\varphi}^2$ as well as of the Hecke operators $T_m$ (for $(m,q)=1$) with eigenvalues $\lambda_{\varphi}(m)$. The Fourier expansion (at $\infty$) of an element $\varphi\in \mathcal{B}(q)$ can be written as
\[\varphi(z) = \sum_{\substack{n = -\infty \\ n\neq 0}}^{\infty} \rho_{\varphi}(n) W_{0,it_{\varphi}}(4\pi \vert n\vert y)e(nx).\]

For technical reasons, it shall prove important to arrange the basis $\mathcal{B}(q)$ conveniently. To do so, we recall that a Hecke--Maa\ss{} cusp form $\varphi\colon \Gamma_0(q)\backslash \mathbb{H}\to \mathbb{C}$ can be ad\`{e}lised (i.e.\ lifted to an ad\`{e}lic automorphic form for $\mathrm{GL}_2(\mathbb{A})$). This ad\`{e}lic lift can in turn be used to generate a cuspidal automorphic representation $\pi_{\varphi}$ of $\mathrm{GL}_2(\mathbb{A})$. Given a cuspidal automorphic representation $\pi$, we write $c(\pi)$ for its conductor. Using this, we can organise our basis of cusp forms as follows:
\[\mathcal{B}(q) = \bigsqcup_{\substack{\pi \textrm{ cusp.\ aut.}\\ c(\pi)\mid q}} \mathcal{B}_{\pi}(q),\]
where $\mathcal{B}_{\pi}(q) = \{ \varphi\in \mathcal{B}(q)\colon \pi = \pi_{\varphi}\}$. By newform theory, we have that $\sharp \mathcal{B}_{\pi}(c(\pi)) = 1$. The (up to scaling by elements in $S^1$) unique element in $\mathcal{B}_{\pi}(c(\pi))$ is called the newform and we denote it by $\varphi_{\pi}$. Without loss of generality, we may assume that $\varphi_{\pi}\in \mathcal{B}(q)$.\footnote{Recall that we take the inner product on $L^2(\Gamma_0(q)\backslash \mathbb{H})$ with respect to the probability Haar measure, so that for each $q'\mid q$, the natural embedding $L^2(\Gamma_0(q')\backslash  \mathbb{H})\hookrightarrow L^2(\Gamma_0(q)\backslash  \mathbb{H})$ is an isometry.}

Given a cuspidal automorphic representation $\pi$ with $c(\pi)\mid q$, the corresponding newform $\varphi_{\pi}$ is an eigenfunction the Hecke operators $T_m$ for all $m\in \mathbb{N}$. To shorten notation, we write $\lambda_{\pi}(m) = \lambda_{\varphi_{\pi}}(m)$ for the corresponding Hecke eigenvalues; similarly, we write $t_{\pi} = t_{\varphi_{\pi}}$ for the spectral parameter. We observe that
\begin{equation}
	\rho_{\varphi_{\pi}}(n) = \sqrt{n}\lambda_{\pi}(n)\cdot \rho_{\varphi_{\pi}}(1) \text{ for }n\in \mathbb{N}.\label{eq:f_h_rel}
\end{equation}
Furthermore, the Rankin--Selberg method produces the relation
\begin{equation}
	\vert \rho_{\varphi_{\pi}}(1)\vert^2 \asymp_q \frac{\cosh(\pi t_{\pi})}{L(1,\pi,\mathrm{ad})}.\label{eq:first_fc}
\end{equation}
Finally, we recall that $L(1,\pi,\mathrm{ad})$ is positive and we have the standard bounds \cite{HL,Li}
\begin{equation}
	T^{-\epsilon}\ll_{q,\epsilon} L(1,\pi,\mathrm{ad}) \ll_{q,\epsilon} T^{\epsilon}.\label{eq:lo_up_bounds}
\end{equation}

A key tool in our estimate shall be played by the following version of the spectral large sieve.

\begin{lemma}\label{lm:spec_larg}
Let $q,N\in \mathbb{N}$ and $T\geq 1$. For a sequence $(a_n)_{n\in \mathbb{N}}$ of complex numbers, we have that
\[\sum_{\substack{\pi \textrm{ cusp.\ aut.}\\ c(\pi)\mid q}} \frac{1}{L(1,\pi,\mathrm{ad})}\left\vert \sum_{N<n\leq 2N} a_n \lambda_{\pi}(n)\right\vert^2 \ll_{q,\epsilon} (T^2+N^{1+\epsilon})\cdot \sum_{N<n\leq N} \vert a_n\vert^2.\]
\end{lemma}
\begin{proof}
We first recall the classical  inequality 
\begin{equation}
	\sum_{\substack{\varphi\in \mathcal{B}(q)\\ \vert t_{\varphi}\vert \leq T}} \frac{1}{\cosh(\pi t_{\varphi})} \left\vert \sum_{N<n\leq 2N}a_n \sqrt{n}\rho_{\varphi}(n) \right\vert^2 \ll_{q,\epsilon} (T^2+N^{1+\epsilon})\cdot \sum_{N<n\leq N} \vert a_n\vert^2, \label{eq: classical_ls}
\end{equation}
as stated, for example, in \cite[Proposition~4.7]{Dr}. To use this, we first apply \eqref{eq:f_h_rel} and \eqref{eq:first_fc} to write
\[\sum_{\substack{\pi \textrm{ cusp.\ aut.}\\ c(\pi)\mid q}} \frac{1}{L(1,\pi,\mathrm{ad})}\left\vert \sum_{N<n\leq 2N} a_n \lambda_{\pi}(n)\right\vert^2 \ll_q \sum_{\substack{\pi \textrm{ cusp.\ aut.}\\ c(\pi)\mid q}} \frac{1}{\cosh(\pi t_{\pi})}\left\vert \sum_{N<n\leq 2N} a_n \rho_{\varphi_{\pi}}(n)\right\vert^2.\]
By positivity, we can extend the sum on the right hand side to a sum over $\mathcal{B}(q)$ and apply \eqref{eq: classical_ls} in order to complete the proof.
\end{proof}

\subsection{CM-forms}

Throughout this section, we let $D<0$ be a fundamental discriminant and consider the imaginary quadratic number field $K=\mathbb{Q}(\sqrt{D})$ with ring of integers $\mathcal{O}_K$. Let $\mathrm{Cl}_D$ denote the class group of $K$ and write $\widehat{\mathrm{Cl}_D}$ for the group of characters of $\mathrm{Cl}_D$. We extend an element $\xi\in \widehat{\mathrm{Cl}_D}$ to a function on ideals $\mathfrak{a}\subseteq \mathcal{O}_K$ in the obvious way and define for $n \in \mathbb{N}$ the multiplicative function
\[\lambda_{\xi}(n) = \sum_{\substack{0\neq \mathfrak{a}\subseteq \mathcal{O}_K\\ \mathrm{Nr}_K(\mathfrak{a})=n}} \xi(\mathfrak{a}).\]
Note that $\lambda_{\xi}(n)\ll_{\epsilon} n^{\epsilon}$ by standard divisor bounds. We also recall that the Hecke $L$-function associated to $\xi$ is given by
\[L(s,\xi) = \sum_{n = 1}^{\infty} \lambda_{\xi}(n)n^{-s} \quad \text{for }\Re(s)>1.\]
We let $\sigma_{\xi}$ denote the automorphic representation obtained from $\xi$ (viewed as a Hecke character over $K$) via automorphic induction. In particular, we have that
\[L(s,\xi) = L(s,\sigma_{\xi}).\]
The classical counterpart to $\sigma_{\xi}$ is the CM-form
\[f_{\xi}(z) = \sum_{n=0}^{\infty} \lambda_{\xi}(n)e(nz),\]
where we have set
\[\lambda_{\xi}(0) = \delta_{\xi = 1} \frac{\sqrt{\vert D\vert}L(1,\chi_D)}{2\pi}.\]
Note that $f_{\xi}$ is a modular form for $\Gamma_0(\vert D\vert)$ of weight one and nebentypus $\chi_D$, the primitive quadratic Dirichlet character modulo $|D|$. The modular form $f_{\xi}$ (and thus the automorphic representation $\sigma_{\xi}$) is cuspidal for complex $\xi$ and is non-cuspidal for real characters (i.e.\ genus characters) $\xi$. In the latter case, $f_{\xi}$ can be expressed in terms of certain Eisenstein series. We refer to \cite[Section~2.3]{Risager} for further discussion.

\subsection{Rankin--Selberg \texorpdfstring{$L$}{L}-functions}

Let $\pi$ and $\sigma$ be automorphic representations of $\mathrm{GL}_2(\mathbb{A})$. At the moment, we do not require that these be cuspidal. In practice, we shall encounter situations where $\sigma$ arises from a certain Eisenstein series. We define the Rankin--Selberg $L$-function $L(s,\pi\otimes \sigma)$ of $\pi$ and $\sigma$ as, for example, in \cite[Section~3]{HM}. For our application below, we shall need a suitable approximate functional equation for this $L$-function. We shall use the following version, which is tailored to our particular set-up.

\begin{lemma}\label{lm:approx_fe}
Let $T\geq 1$ and $q\in \mathbb{N}$ and let $D<0$ be a fundamental discriminant. Let $\pi$ be a cuspidal automorphic representation of $\mathrm{PGL}_2(\mathbb{A})$ with $c(\pi)\mid q$ and $T\leq t_{\pi}\leq 2T$ and let $\xi\in \widehat{\mathrm{Cl}_D}$. Then
\[L\left(\frac{1}{2},\pi\otimes \sigma_{\xi}\right) \ll_{D,q,\epsilon} T^{\epsilon}\max_{N\leq T^{2+\epsilon}}\frac{\vert \mathcal{S}_{\pi,\xi}(N)\vert}{\sqrt{N}}  + T^{\epsilon},\]
where
\begin{equation}
	\mathcal{S}_{\pi,\xi}(N) = \sum_{n = 1}^{\infty} \lambda_{\pi}(n)\lambda_{\xi}(n) W\left(\frac{n}{N}\right) \label{eq:def_S}
\end{equation}
with $W$ a smooth function supported in $[\frac{1}{2},\frac{5}{2}]$ satisfying $W(x)\ll 1$.
\end{lemma}
\begin{proof}
When $\sigma_{\xi}$ is cuspidal, this follows directly from \cite[Lemma~2.1]{Bl_RS} after inserting a smooth partition of unity. The modifications necessary to deal with non-cuspidal $\xi$ are straightforward, since the $L$-function factors in this case.
\end{proof}

Note that, since our applications requires averaging over $\pi$, it is crucial that the smooth weight $W(\cdot)$ is independent of $\pi$. This comes at the cost of a relatively weak error term $O_{D,q,\epsilon}(T^{\epsilon})$, which will be sufficient for us. The standard approximate functional equation comes with a much better error term, but introduces non-trivial $\pi$ dependence in $W$. See for example \cite[(50)]{HM} or \cite[Section~3.1]{Risager} for statements of such approximate functional equations.

\subsection{Heegner points and Waldspurger's formula}\label{sec:Heegner}

We follow \cite{GKZ} to define the notion of Heegner points; see also \cite{BBK} for a glimpse of the ad\`{e}lic theory. Fix $q\in \mathbb{N}$ and let $D<0$ be a fundamental discriminant. Let $K=\mathbb{Q}(\sqrt{D})$. We assume the following conditions:
\begin{itemize}
	\item if $p\mid q$, then $p$ is either split or ramified in $K$;
	\item if $p^2\mid q$, then $p$ is split in $K$.
\end{itemize}
These conditions allow us to fix $r\in \mathbb{Z}/2q\mathbb{Z}$ such that 
\[r^2\equiv D \hspace{-.25cm} \pmod{4q}.\]
Following \cite{GKZ}, we define
\[m=\left( q,r,\frac{r^2-D}{4q}\right).\]
This is well defined even though $r$ is only defined modulo $2q$.

Let $\mathcal{Q}_{q,D}$ denote the set of integral binary quadratic forms $Q(x,y)=Ax^2+Bxy+Cy^2$ of discriminant $D$ and such that $A>0$, $A \equiv 0 \pmod{q}$ and $B\equiv r \pmod{2q}$. We also write $Q=[A,B,C]$, so that
\[\mathcal{Q}_{q,D}  = \{ [A,B,C]\colon B^2-4AC=D,\ A > 0,\ A \equiv 0 \hspace{-.25cm} \pmod{q},\ B\equiv r \hspace{-.25cm} \pmod{2q}\}.\]
We write $\mathcal{Q}_{q,D}^{\circ}$ for the subset of primitive $Q$ (i.e\ $Q=[A,B,C]\in \mathcal{Q}_{q,D}$ for which $(A/q,B,C)=1$). Note that $\Gamma_0(q)$ acts on $\mathcal{Q}_{q,D}^{\circ}$ from the left via
\[[\gamma Q](x,y) = Q(dx-by,ay-cx) \text{ for }\gamma=\left(\begin{matrix} a & b \\ c & d \end{matrix}\right)\in \Gamma_0(q).\]
Finally, given $Q=[A,B,C]\in \mathcal{Q}_{q,D}^{\circ}$, we associate the points
\[z_Q = \frac{-B+i\sqrt{\vert D\vert}}{2A}\in \mathbb{H}.\]
We note that $\gamma z_Q = z_{\gamma Q}$. A point $\Gamma_0(q)z_Q\in \Gamma_0(q)\backslash \mathbb{H}$ for $Q\in \mathcal{Q}_{q,D}$ is called a Heegner point of discriminant $D$ and level $q$. By abuse of notation, we continue to write $z_Q$ in place of $\Gamma_0(q)z_Q$, with the understanding that such a point is considered up to $\Gamma_0(q)$-equivalence.

After this general discussion, we briefly specialise to Heegner points of level $1$, which are arguably the most studied points. We write
\[\mathcal{H}_D = \{ z_Q\colon Q\in \mathcal{Q}_{1,D}^{\circ}\}.\]
It is well known that 
\[\sharp \mathcal{H}_D = \sharp \mathcal{Q}_{1,D}^{\circ} = \sharp \mathrm{Cl}_D.\]
The equalities between these cardinalities arise from natural one-to-one correspondences, which we use to associate $z_{\mathfrak{a}}\in \mathcal{H}_D$ to an ideal (class) $\mathfrak{a}\in \mathrm{Cl}_D$.

Let us now return to the general discussion. Given $Q=[qA,B,C]\in \mathcal{Q}_{q,D}^{\circ}$, we let
\[m_1(Q)=(q,B,A) \text{ and }m_2(Q)=(q,B,C).\]
One checks that $m=m_1(Q)m_2(Q)$ and $(m_1(Q),m_2(Q))=1$. We further split $q=q_1q_2$ with $(q_1,m_2(Q)) = (q_2,m_1(Q))=1$.

Let $m=m_1m_2$ with $(m_1,m_2)=1$. It is shown on \cite[p.505]{GKZ} that the map $Q\mapsto \widetilde{Q}=[q_1A,B,q_2C]$ defines a one-to-one correspondence
\[\Gamma_0(q)\backslash \{ Q\in \mathcal{Q}_{q,D}^{\circ}\colon  m_1(Q)=m_1 \text{ and }m_2(Q)=m_2 \} \longleftrightarrow \mathrm{SL}_2(\mathbb{Z})\backslash \mathcal{Q}_{1,D}^{\circ}.\]
This allows for the following construction: given $Q\in \mathcal{Q}_{q,D}^{\circ}$, we can write
\begin{equation}
	z_{Q} = g_{q_2} \cdot z_{\mathfrak{a}_0}, \label{eq:homogenisation}
\end{equation}
where $\mathfrak{a}_0\in \mathrm{Cl}_D$ is such that $z_{\mathfrak{a}_0} = z_{\widetilde{Q}}\in \mathcal{H}_D$ and $g_{q_2} = \mathrm{diag}(1,q_2)$.

\begin{example}\label{example}
If we take $q=1$ and $D=-4$, then
\[\Gamma\backslash \mathcal{Q}_{1,-4}^+ = \{ [1,0,1] \}.\]
Thus, $i\in \Gamma\backslash \mathbb{H}$ is the only Heegner point of discriminant $-4$. In this case, we have $\sharp \Gamma_{[1,0,1]} = 2$.

If we take $q=2$ and $D=-4$, then one easily checks that the condition proposed above is satisfied. We choose $r=2$ and observe that
\[\Gamma_0(2)\backslash \mathcal{Q}_{2,-4}^+ = \{[2,2,1]\}.\]
In particular, $\frac{-1 + i}{2}\in \Gamma_0(2)\backslash \mathbb{H}$ is the unique Heegner point of discriminant $-4$ and level $2$. Note that 
\[[2,2,1] = \left(\begin{matrix} 0 & -1 \\ 1 & 1 \end{matrix}\right)[1,0,1].\]
However, we can also write 
\[\frac{-1 + i}{2}=g_2 \cdot z_{[1,2,2]}\]
for $g_2=\mathrm{diag}(1,2)$. Note that $z_{[1,2,2]} = -1+i\in \mathcal{H}_{-4}$ is of course equivalent to $i=z_{[1,0,1]}$.
\end{example}

\begin{lemma}\label{lm:walds} 
Let $q\in \mathbb{N}$, $D<0$ a fundamental discriminant, and $T\geq 1$. Let $\pi$ be a cuspidal automorphic representation of $\mathrm{PGL}_2(\mathbb{A})$ for which $c(\pi) \mid q$ and $T\leq t_{\pi}\leq 2T$. For a Heegner point $z_0\in \Gamma_0(q)\backslash\mathbb{H}$ of discriminant $D$ and level $q$ and for $\varphi\in \mathcal{B}_{\pi}(q)$, we have that
\[\vert \varphi(z_0)\vert^2 \ll_{D,q} \sum_{\xi\in \widehat{\mathrm{Cl}_D}}\frac{L(\frac{1}{2},\pi\otimes \sigma_{\xi})}{L(1,\pi,\mathrm{ad})}.\]
\end{lemma}

Here we note that $L(\frac{1}{2},\pi\otimes \sigma_{\xi})$ is nonnegative due to the work of Waldspurger \cite{Wa}.

\begin{proof}
By newform theory, we may write $\varphi$ as a finite linear combination of oldforms, where the number of summands and the coefficients in this linear combination are bounded independently of $T$; see \cite[Section 3]{BM} for details. We may therefore assume without loss of generality that $\varphi$ is an oldform.

We now note that, according to our assumption that $z_0$ is a Heegner point of discriminant $D$ and level $q$, we can use \eqref{eq:homogenisation} to find $q_2\mid q$ (with $(q_2,q/q_1)=1$) and $\mathfrak{a}\in \mathrm{Cl}_D$ such that
\[z_0 = g_{q_2} \cdot z_{\mathfrak{a}_0}.\]
By positivity and orthogonality of characters,
\[\vert \varphi(z_0)\vert^2 \leq \sum_{\mathfrak{a}\in \mathrm{Cl}_D} \vert \varphi(g_{q_2} \cdot z_{\mathfrak{a}})\vert^2 = \frac{1}{\sharp\mathrm{Cl}_D} \sum_{\xi\in \widehat{\mathrm{Cl}_D}} \left\vert \sum_{\mathfrak{a}\in \mathrm{Cl}_D} \varphi(g_{q_2} \cdot z_{\mathfrak{a}})\xi(\mathfrak{a})\right\vert^2.\]
The result then follows once we have deduced the bound
\begin{equation}
	\left\vert \sum_{\mathfrak{a}\in \mathrm{Cl}_D} \varphi(g_{q_2} \cdot z_{\mathfrak{a}})\xi(\mathfrak{a})\right\vert^2 \ll_{D,q} \frac{L(\frac{1}{2},\pi\otimes \sigma_{\xi})}{L(1,\pi,\mathrm{ad})}\label{eq:toshow}
\end{equation}
from Waldspurger's formula.

Showing \eqref{eq:toshow} is best done ad\`{e}lically. Let $\xi_{\mathbb{A}}$ denote the Hecke character corresponding to $\xi$ and let $\varphi_{\mathbb{A}}\in \pi$ be the ad\`{e}lic lift of $\varphi$. We refer to \cite[Section~4.3]{HN} for details. Since $\varphi$ is assumed to be an oldform, the ad\`{e}lic lift $\varphi_{\mathbb{A}}$ corresponds to a pure tensor under the isomorphism $\pi\cong \otimes_v \pi_v$. We can now write the left hand side of \eqref{eq:toshow} as an ad\`{e}lic toric period integral
\[\sum_{\mathfrak{a}\in \mathrm{Cl}_D} \varphi(g_{q_2} \cdot z_{\mathfrak{a}})\xi(\mathfrak{a}) = C_D\cdot \int_{T(\mathbb{Q})Z(\mathbb{A})\backslash T(\mathbb{A})} \xi_{\mathbb{A}}(t)^{-1}\varphi_{\mathbb{A}}(tg) \, d\mu_T(t).\]
We refer to \cite[Sections~2.1 and~2.5]{BB} for details. In particular $T\subseteq \mathrm{GL}_2$ is an anisotropic torus arising from an optimal embedding $\mathcal{O}_{K} \hookrightarrow \mathrm{Mat}_{2\times 2}(\widehat{\mathbb{Z}})$ and $g\in \mathrm{GL}_2(\mathbb{A})$ is some matrix depending on this embedding and $g_{q_2}$. If $S = \{\infty\} \cup \{p\mid qD\}$, then the components of $g$ at $v\not\in S$ are trivial. Also, note that the constant $C_D\in \mathbb{R}_{>0}$ depends on the measure normalisations and can be given explicitly. However, since we do not keep track of the $D$-dependence, we may ignore this here with impunity.

We now apply Waldspurger's formula \cite[Proposition~7]{Wa}; see also \cite[(10.2)]{BBK} for a version with $\xi_{\mathbb{A}}$ trivial. We obtain
\[\left\vert \sum_{\mathfrak{a}\in \mathrm{Cl}_D} \varphi(g_{q_2} \cdot z_{\mathfrak{a}})\xi(\mathfrak{a}) \right\vert^2 = C_{D,S}'\cdot I_{S,D}(g \cdot \varphi_{\mathbb{A}},\xi_{\mathbb{A}})\cdot \frac{L^S(\frac{1}{2},\pi\otimes \sigma_{\xi})}{L^S(1,\pi,\mathrm{ad})}.\]
Here $C_{D,S}'\in \mathbb{R}_{>0}$ is a new constant dependent only on $D$ and $S$ arising from local measure normalisations at each element of $S$. Furthermore,  $I_{S,D}(g \cdot \varphi_{\mathbb{A}},\xi_{\mathbb{A}})$ is a product of local integrals. Since we are not interested on dependencies in $q$ and $D$ we can estimate $C_{D,S}'\ll_{q,D} 1$ and complete $\frac{L^S(\frac{1}{2},\pi\otimes \sigma_{\xi})}{L^S(1,\pi,\mathrm{ad})}$ to $\frac{L(\frac{1}{2},\pi\otimes \sigma_{\xi})}{L(1,\pi,\mathrm{ad})}$ by artificially inserting the missing local factors. Thus, since the $L$-values appearing on the right hand side of \eqref{eq:toshow} are nonnegative, it is sufficient to show that 
\[I_{S,D}(g \cdot \varphi_{\mathbb{A}},\xi_{\mathbb{A}}) \ll_{q,D} 1.\]
This can be achieved by estimating the local integrals trivially essentially following the proof of \cite[Lemma~10.3]{BBK}, or alternatively by appealing to the calculation of the exact values of these local integrals determined in \cite{MW}. We omit the details.
\end{proof}

\section{Proof of Theorem~\ref{th:4th_heegner}}\label{sec:main}

We are now in a position to prove Theorem~\ref{th:4th_heegner}.

\begin{proof}[Proof of Theorem~\ref{th:4th_heegner}]
Recall we have fixed $q\in \mathbb{N}$ and that $z_0\in \Gamma_0(q)\backslash \mathbb{H}$ is a Heegner point of fundamental discriminant $D<0$.

We start by applying Waldspurger's formula as presented in Lemma~\ref{lm:walds}. Together with the Cauchy--Schwarz inequality, we thereby obtain
\begin{align*}
	\sum_{\substack{\varphi\in \mathcal{B}(q)\\ T\leq t_{\varphi} \leq  2T}}\vert \varphi(z_0)\vert^4 &\ll_{D,q} \sum_{\substack{\pi \textrm{ cusp.\ aut.}\\ c(\pi)\mid q\\ T\leq t_{\pi}\leq 2T}} \left(\sum_{\xi\in \widehat{\mathrm{Cl}_D}}\frac{L(\frac{1}{2},\pi\otimes \sigma_{\xi})}{L(1,\pi,\mathrm{ad})}\right)^2 \\
	&\ll_{D,q} \max_{\xi\in \widehat{\mathrm{Cl}_D}}\sum_{\substack{\pi \textrm{ cusp.\ aut.}\\ c(\pi)\mid q\\ T\leq t_{\pi}\leq 2T}}\frac{L(\frac{1}{2},\pi\otimes \sigma_{\xi})^2}{L(1,\pi,\mathrm{ad})^2}.
\end{align*}
At this point, we invoke the bound \eqref{eq:lo_up_bounds} for $L(1,\pi,\mathrm{ad})$ and insert the approximate functional equation for $L(\frac{1}{2},\pi\otimes \sigma_{\xi})$ as stated in Lemma~\ref{lm:approx_fe}. We arrive at
\begin{equation}
	\sum_{\substack{\varphi\in \mathcal{B}(q)\\ T\leq t_{\varphi} \leq  2T}}\vert \varphi(z_0)\vert^4  \ll_{D,q,\epsilon}\max_{\xi\in \widehat{\mathrm{Cl}_D}}\max_{N\leq T^{2+\epsilon}}\frac{T^{\epsilon}}{N}\sum_{\substack{\pi \textrm{ cusp.\ aut.}\\ c(\pi)\mid q\\ T\leq t_{\pi}\leq 2T}}\frac{\vert \mathcal{S}_{\pi,\xi}(N)\vert^2}{L(1,\pi,\mathrm{ad})} + T^{2+\epsilon}.\label{eq:above} 
\end{equation}
After recalling the definition of $\mathcal{S}_{\pi,\xi}(N)$ from \eqref{eq:def_S}, we observe that the spectral large sieve as presented in Lemma~\ref{lm:spec_larg} is applicable. We find that
\[\sum_{\substack{\pi \textrm{ cusp.\ aut.}\\ c(\pi)\mid q\\ T\leq t_{\pi}\leq 2T}}\frac{\vert \mathcal{S}_{\pi,\xi}(N)\vert^2}{L(1,\pi,\mathrm{ad})} \ll_{D,q,\epsilon} (NT)^{\epsilon}(T^2+N)\sum_{n = 1}^{\infty} \left\vert \lambda_{\xi}(n) W\left(\frac{n}{N}\right)\right\vert^2.\]
Since $W(\cdot)$ is $1$-bounded and compactly supported and because $\lambda_{\xi}(n)\ll_{\epsilon} n^{\epsilon}$, we obtain 
\[\sum_{\substack{\pi \textrm{ cusp.\ aut.}\\ c(\pi)\mid q\\ T\leq t_{\pi}\leq 2T}}\frac{\vert \mathcal{S}_{\pi,\xi}(N)\vert^2}{L(1,\pi,\mathrm{ad})} \ll_{D,q,\epsilon} (NT)^{\epsilon}(NT^2+N^2).\]
Inserting this estimate in \eqref{eq:above} above yields the desired bound
\[\sum_{\substack{\varphi\in \mathcal{B}(q)\\ T\leq t_{\varphi} \leq  2T}}\vert \varphi(z_0)\vert^4  \ll_{D,q,\epsilon}T^{2+\epsilon}.\qedhere\]
\end{proof}

\section{Proof of Theorem~\ref{thm:twistedsecondmoment}}
\label{sec:pretrace}

Let $z_0 = \frac{-B + i\sqrt{|D|}}{2A}$, where $(A,B,C) \in \mathbb{Z}^3$ with $(A,B,C) = 1$, $D = B^2 - 4AC$ is a negative discriminant, and $A > 0$. Let $\ell \in \mathbb{N}$ and $T \geq 1$. We consider the spectral twisted second moment
\begin{equation}
\label{eqn:spectraltwistedsecondmoment}
\sum_{\varphi\in \mathcal{B}(1)} \lambda_{\varphi}(\ell) \cdot \vert \varphi(z_0)\vert^2 e^{-\frac{t_{\varphi}^2}{T^2}} + \frac{1}{4\pi} \int_{-\infty}^{\infty} \lambda(\ell,t) \cdot \left\vert E\left(z_0,\frac{1}{2} + it\right)\right\vert^2 e^{-\frac{t^2}{T^2}} \, dt.
\end{equation}

\begin{lemma}
\label{lem:pretrace}
The spectral twisted second moment \eqref{eqn:spectraltwistedsecondmoment} is equal to
\begin{equation}
\label{eqn:partialsummation}
-\frac{1}{\sqrt{\ell}} \int_{0}^{\infty} k'(\delta) \sharp\{\gamma \in \Gamma_{\ell} : u(\gamma z_0,z_0) \leq \delta\} \, d\delta - \frac{3}{\pi} e^{\frac{1}{4T^2}} \frac{1}{\sqrt{\ell}} \sum_{ad = \ell} d,
\end{equation}
where $k : [0,\infty) \to \mathbb{C}$ denotes the inverse Selberg--Harish-Chandra transform of $h(t) = e^{-t^2/T^2}$,
\[u(z,w) = \frac{|z - w|^2}{4\Im(z) \Im(w)}, \qquad \Gamma_{\ell} = \left\{\begin{pmatrix} a & b \\ c & d \end{pmatrix} : a,b,c,d \in \mathbb{Z}, \ ad - bc = \ell\right\}.\]
\end{lemma}

\begin{proof}
By the pretrace formula \cite[Theorem~7.4]{Iw}, the spectral twisted second moment \eqref{eqn:spectraltwistedsecondmoment} is equal to
\begin{equation}
\label{eqn:Tellku}
T_{\ell}\big\vert_{z = z_0} \sum_{\gamma \in \Gamma} k(u(\gamma z_0,z)) - \frac{3}{\pi} e^{\frac{1}{4T^2}} T_{\ell} 1,
\end{equation}
where $T_{\ell}$ denotes the $\ell$-th Hecke operator and the second term arises from the contribution of the constant function in the spectral decomposition of $L^2(\Gamma \backslash \mathbb{H})$. Since the $\ell$-th Hecke operator acts on $\Gamma$-invariant functions $f : \mathbb{H} \to \mathbb{C}$ via
\[(T_{\ell} f)(z) = \frac{1}{\sqrt{\ell}} \sum_{\gamma \in \Gamma \backslash \Gamma_{\ell}} f(\gamma z) = \frac{1}{\sqrt{\ell}} \sum_{ad = n} \sum_{b \hspace{-.25cm} \pmod{d}} f\left(\frac{az + b}{d}\right),\]
we may unfold in order to see that \eqref{eqn:Tellku} is equal to
\[\frac{1}{\sqrt{\ell}} \sum_{\gamma \in \Gamma_{\ell}} k(u(\gamma z_0,z_0)) - \frac{3}{\pi} e^{\frac{1}{4T^2}} \frac{1}{\sqrt{\ell}} \sum_{ad = \ell} d.\]
The desired identity subsequently follows via partial summation.
\end{proof}

To estimate \eqref{eqn:partialsummation}, we use the following two results.

\begin{prop}
\label{prop:k'bound}
Let $k : [0,\infty) \to \mathbb{C}$ denote the inverse Selberg--Harish-Chandra transform of $h(t) = e^{-t^2/T^2}$. Then
\begin{equation}
\label{eqn:k'bound}
k'(\delta) \ll \begin{dcases*}
T^4 & for $\delta \leq \sinh^2\frac{1}{T}$,	\\
T^4 e^{-T^2 (\arsinh \sqrt{\delta})^2} & for $\sinh^2\frac{1}{T} \leq \delta \leq 1$,	\\
T^4 e^{-T^2 (\arsinh \sqrt{\delta})^2} \frac{(\arsinh \sqrt{\delta})^{3/2}}{\delta^{3/2}} & for $\delta \geq 1$.
\end{dcases*}
\end{equation}
\end{prop}

\begin{prop}
\label{prop:Gammaellubound}
Let $z_0$ be as above. For $\delta \geq 0$,
\begin{equation}
\label{eqn:Gammaellubound}
\sharp\{\gamma \in \Gamma_{\ell} : u(\gamma z_0,z_0) \leq \delta\} \ll_{D,\epsilon} \begin{dcases*}
\ell^{\epsilon} & if $0 \leq \delta \leq \frac{1}{\ell}$,	\\
(\ell \delta)^{1 + \epsilon} & if $\delta \geq \frac{1}{\ell}$.
\end{dcases*}
\end{equation}
\end{prop}

\begin{rem}
For \emph{arbitrary} $z_0 \in \Hb$, the weaker bound
\[\sharp\{\gamma \in \Gamma_{\ell} : u(\gamma z_0,z_0) \leq \delta\} \ll_{z_0,\epsilon} \begin{dcases*}
\ell^{\epsilon} & if $0 \leq \delta \leq \frac{1}{\ell^4}$,	\\
\ell^{1 + \epsilon} \delta^{1/4} & if $\frac{1}{\ell^4} \leq \delta \leq 1$,	\\
(\ell \delta)^{1 + \epsilon} & if $\delta \geq 1$
\end{dcases*}\]
is shown in \cite[Appendix 1]{IS}.
\end{rem}

With these in hand, we may complete the proof of Theorem~\ref{thm:twistedsecondmoment}.

\begin{proof}[Proof of Theorem~\ref{thm:twistedsecondmoment}]
By Lemma~\ref{lem:pretrace}, the spectral twisted second moment \eqref{eqn:spectraltwistedsecondmoment} is equal to
\[-\frac{1}{\sqrt{\ell}} \int_{0}^{\infty} k'(\delta) \sharp\{\gamma \in \Gamma_{\ell} : u(\gamma z_0,z_0) \leq \delta\} \, d\delta - \frac{3}{\pi} e^{\frac{1}{4T^2}} \frac{1}{\sqrt{\ell}} \sum_{ad = \ell} d.\]
Inputting the bounds \eqref{eqn:k'bound} and \eqref{eqn:Gammaellubound} for the former term and standard divisor bounds for the latter term, we see that this is
\[\ll_{D,\epsilon} \begin{dcases*}
\frac{T^{2 + \epsilon}}{\sqrt{\ell}} & if $\ell \leq T^2$,	\\
\ell^{\frac{1}{2} + \epsilon} & if $\ell \geq T^2$,
\end{dcases*}\]
as desired.
\end{proof}

\subsection{Proof of Proposition~\ref{prop:k'bound}}

The proof of Proposition~\ref{prop:k'bound} is based on a similar calculation for $k(u)$ given in \cite[Lemma~5.1]{CR}.

\begin{proof}[Proof of Proposition~\ref{prop:k'bound}]
For $h(t) = e^{-t^2/T^2}$, we have via the identity \cite[(1.63)]{Iw} for $k(u)$ and integration by parts that
\begin{align*}
k(u) & = \frac{1}{8\pi^2} \int_{2\arsinh \sqrt{u}}^{\infty} \frac{1}{\sqrt{\sinh^2 \frac{v}{2} - u}} \int_{-\infty}^{\infty} e^{-\frac{t^2}{T^2}} t \sin(tv) \, dt \, dv	\\
& = \frac{T^3}{16\pi^{3/2}} \int_{2 \arsinh \sqrt{u}}^{\infty} \frac{v e^{-\frac{T^2 v^2}{4}}}{\sqrt{\sinh^2 \frac{v}{2} - u}} \, dv	\\
& = \frac{T^3}{8\pi^{3/2}} \int_{2 \arsinh \sqrt{u}}^{\infty} \sqrt{\sinh^2 \frac{v}{2} - u} \frac{(T^2 v^2 + 2v \coth v - 2) e^{-\frac{T^2 v^2}{4}}}{\sinh v} \, dv.
\end{align*}
It follows that
\begin{align*}
-k'(u) & = \frac{T^3}{16\pi^{3/2}} \int_{2 \arsinh \sqrt{u}}^{\infty} \frac{1}{\sqrt{\sinh^2 \frac{v}{2} - u}} \frac{(T^2 v^2 + 2v \coth v - 2) e^{-\frac{T^2 v^2}{4}}}{\sinh v} \, dv	\\
& = K_1 + K_2,
\end{align*}
where for $u_0 \in [u,\infty)$ to be chosen,
\begin{align*}
K_1 & = \frac{T^3}{4\pi^{3/2}} \int_{2 \arsinh \sqrt{u}}^{2 \arsinh \sqrt{u_0}} \frac{(T^2 v^2 + 2v \coth v - 2) e^{-\frac{T^2 v^2}{4}}}{\sinh^2 v} \frac{d}{dv} \sqrt{\sinh^2 \frac{v}{2} - u} \, dv,	\\
K_2 & = -\frac{T^3}{32\pi^{3/2}} \int_{2 \arsinh \sqrt{u_0}}^{\infty} \frac{1}{\sqrt{\sinh^2 \frac{v}{2} - u}} \frac{d}{dv} \frac{v e^{-\frac{T^2 v^2}{4}}}{\sinh v} \, dv
\end{align*}
For the former integral,
\begin{align*}
K_1 & \leq \frac{T^3}{2\pi^{3/2} u(u + 1)} \left(T^2 (\arsinh\sqrt{u})^2 + \arsinh\sqrt{u} \frac{2u + 1}{\sqrt{u(u + 1)}} - 1\right)	\\
& \hspace{3cm} \times e^{-T^2(\arsinh\sqrt{u})^2} \int_{2\arsinh\sqrt{u}}^{2 \arsinh \sqrt{u_0}} \frac{d}{dv} \sqrt{\sinh^2 \frac{v}{2} - u} \, dv	\\
& = \frac{T^3}{2\pi^{3/2}u(u + 1)} \left(T^2 (\arsinh\sqrt{u})^2 + \arsinh\sqrt{u} \frac{2u + 1}{\sqrt{u(u + 1)}} - 1\right)    \\
& \hspace{3cm} \times e^{-T^2(\arsinh\sqrt{u})^2} \sqrt{u_0 - u}.
\end{align*}
Similarly, for the latter,
\begin{align*}
K_2 & \leq -\frac{T^3}{32\pi^{3/2}} \frac{1}{\sqrt{u_0 - u}} \int_{2\arsinh\sqrt{u}}^{\infty} \frac{d}{dv} \frac{v e^{-\frac{T^2 v^2}{4}}}{\sinh v} \, dv	\\
& = \frac{T^3}{8\pi^{3/2}} \frac{e^{-T^2(\arsinh\sqrt{u})^2} \arsinh\sqrt{u}}{\sqrt{u(u + 1)(u_0 - u)}}.
\end{align*}
We choose
\[u_0 = u + \frac{\sqrt{u(u + 1)} \arsinh\sqrt{u}}{4\left(T^2 (\arsinh\sqrt{u})^2 + \arsinh\sqrt{u} \frac{2u + 1}{\sqrt{u(u + 1)}} - 1\right)},\]
so that
\begin{multline*}
-k'(u) \leq \frac{T^3}{8\pi^{3/2}} \frac{e^{-T^2(\arsinh\sqrt{u})^2}}{(u(u + 1))^{3/4}} \\
\times\left(\left(T^2 (\arsinh\sqrt{u})^2 + \arsinh\sqrt{u} \frac{2u + 1}{\sqrt{u(u + 1)}} - 1\right)\arsinh\sqrt{u}\right)^{\frac{1}{2}}.
\end{multline*}
This implies the desired bounds for $k'(u)$.
\end{proof}

\subsection{Proof of Proposition~\ref{prop:Gammaellubound}}

We begin by expressing the quantity $u(\gamma z_0,z_0)$ in terms of quaternary quadratic forms in the entries of $\gamma$.

\begin{lemma}
\label{lem:uquadform}
Let $z_0 = \frac{-B + i\sqrt{|D|}}{2A}$, where $(A,B,C) \in \mathbb{Z}^3$ with $(A,B,C) = 1$, $D = B^2 - 4AC$ is a negative discriminant, and $A > 0$. For $g = \begin{psmallmatrix} a & b \\ c & d \end{psmallmatrix} \in \GL_2^+(\mathbb{R})$, we have that
\[u(g z_0,z_0) = \frac{Q_1(a,b,c,d)}{|D|Q_2(a,b,c,d)},\]
where $Q_1$ denotes the positive semi-definite integral quaternary quadratic form
\begin{multline*}
Q_1(a,b,c,d) = AC a^2 - AB ab + BC ac - 2ACad	\\
+ A^2 b^2 + (2AC - B^2)bc + ABbd + C^2 c^2 - BC cd + ACd^2
\end{multline*}
and $Q_2$ denotes the isotropic integral quaternary quadratic form $Q_2(a,b,c,d) = ad - bc$.
\end{lemma}

\begin{proof}
We have that
\[g z_0 = \frac{a \left(\frac{-B + i\sqrt{|D|}}{2A}\right) + b}{c \left(\frac{-B + i\sqrt{|D|}}{2A}\right) + d} = \frac{2Abd - B(ad + bc) + 2Cac + i\sqrt{|D|}(ad - bc)}{2(Ad^2 - Bcd + Cc^2)}.\]
Thus
\begin{align*}
u(g z_0,z) & = \frac{|g z_0 - z_0|^2}{4\Im(g z_0) \Im(z_0)}	\\
& = \frac{\left(\frac{2Abd - B(ad + bc) + 2Cac}{2(Ad^2 - Bcd + Cc^2)} + \frac{B}{2A}\right)^2 + \left(\frac{ad - bc}{2(Ad^2 - Bcd + Cc^2)} - \frac{1}{2A}\right)^2 (4AC - B^2)}{\frac{|D|(ad - bc)}{A (Ad^2 - Bcd + Cc^2)}}	\\
& = \frac{ACa^2 - ABab + BCac - 2ACad}{|D|(ad - bc)}	\\
& \qquad + \frac{A^2 b^2 + (2AC - B^2)bc + ABbd + C^2 c^2 - BCcd + ACd^2}{|D|(ad - bc)}.
\qedhere
\end{align*}
\end{proof}

\begin{cor}
With notation as in Lemma~\ref{lem:uquadform}, we have that for $\delta > 0$,
\begin{multline*}
\sharp\{\gamma \in \Gamma_{\ell} : u(\gamma z_0,z_0) \leq \delta\}	\\
= \sum_{0 \leq n \leq |D|\ell \delta} \sharp\left\{(a,b,c,d) \in \mathbb{Z}^4 : Q_1(a,b,c,d) = n, \ Q_2(a,b,c,d) = \ell\right\}.
\end{multline*}
\end{cor}

We give an upper bound for this cardinality by simultaneously block diagonalising this pair of quaternary quadratic forms. The resulting $2 \times 2$ symmetric matrices are associated to the positive definite integral binary quadratic form $Q_3(a,b) = Aa^2 - Bab + Cb^2$. At this stage, it is beneficial to associate symmetric matrices $A_1,A_2,A_3$ to the quadratic forms $Q_1,Q_2,Q_3$, namely
\begin{gather*}
A_1 = \begin{pmatrix} 2AC & -AB & BC & -2AC \\ -AB & 2A^2 & 2AC - B^2 & AB \\ BC & 2AC - B^2 & 2C^2 & -BC \\ -2AC & AB & -BC & 2AC \end{pmatrix}, \qquad  A_2 = \begin{pmatrix} 0 & 0 & 0 & 1 \\ 0 & 0 & -1 & 0 \\ 0 & -1 & 0 & 0 \\ 1 & 0 & 0 & 0 \end{pmatrix},   \\
A_3 = \begin{pmatrix} 2A & -B \\ -B & 2C \end{pmatrix},
\end{gather*}
so that
\begin{gather*}
Q_1(a,b,c,d) = \frac{1}{2} \begin{pmatrix} a & b & c & d \end{pmatrix} A_1 \begin{pmatrix} a \\ b \\ c \\ d \end{pmatrix}, \quad Q_2(a,b,c,d) = \frac{1}{2} \begin{pmatrix} a & b & c & d \end{pmatrix} A_2 \begin{pmatrix} a \\ b \\ c \\ d \end{pmatrix},    \\
Q_3(a,b) = \frac{1}{2} \begin{pmatrix} a & b \end{pmatrix} A_3 \begin{pmatrix} a \\ b \end{pmatrix}.
\end{gather*}

\begin{lemma}
We have that
\begin{multline}
\label{eqn:quadformCOV}
\sharp\left\{(a,b,c,d) \in \mathbb{Z}^4 : Q_1(a,b,c,d) = n, \ Q_2(a,b,c,d) = \ell\right\}	\\
= \sharp\left\{(a',b',c',d') \in \mathbb{Z}^4 : Q_3(c',d') = A|D|n, \ Q_3(a',b') = A|D|(n + |D|\ell),\right.	\\
\left. Aa' - Bb' - Ac' \equiv -Cb' - Bc' + Cd' \equiv A(b' + d') \equiv A(a' + c') \equiv 0 \hspace{-.25cm} \pmod{A|D|}\right\}.
\end{multline}
\end{lemma}

\begin{proof}
The matrix
\[A_2^{-1} A_1 = \begin{pmatrix} -2AC & AB & -BC & 2AC \\ -BC & B^2 - 2AC & -2C^2 & BC \\ AB & -2A^2 & B^2 - 2AC & -AB \\ 2AC & -AB & BC & -2AC \end{pmatrix}\]
is diagonalisable, namely $A_2^{-1} A_1 = S \diag(0,0,D,D) S^{-1}$ with
\begin{align*}
S & = \frac{1}{A|D|} \begin{pmatrix} A & -B & -A & 0 \\ 0 & -C & -B & C \\ 0 & A & 0 & A \\ A & 0 & A & 0 \end{pmatrix},    \\
S^{-1} & = \begin{pmatrix} 2AC & -AB & BC & 2AC - B^2 \\ AB & -2A^2 & 2AC & -AB \\ -2AC & AB & -BC & 2AC \\ -AB & 2A^2 & 2AC - B^2 & AB \end{pmatrix}.
\end{align*}
Moreover, we have that
\begin{align*}
A_1' & = S^T A_1 S = \frac{1}{A|D|} \begin{pmatrix} 0 & 0 & 0 & 0 \\ 0 & 0 & 0 & 0 \\ 0 & 0 & 2A & -B \\ 0 & 0 & -B & 2C \end{pmatrix},	\\
A_2' & = S^T A_2 S = \frac{1}{A|D|^2} \begin{pmatrix} 2A & -B & 0 & 0 \\ -B & 2C & 0 & 0 \\ 0 & 0 & -2A & B \\ 0 & 0 & B & -2C \end{pmatrix}.
\end{align*}
With this in mind, we let
\[\begin{pmatrix} a' \\ b' \\ c' \\ d' \end{pmatrix} = S^{-1} \begin{pmatrix} a \\ b \\ c  \\ d \end{pmatrix} = \begin{pmatrix} 2ACa - ABb + BCc + (2AC - B^2)d \\ A(Ba - 2Ab + 2Cc - Bd) \\ -2ACa + ABb - BCc + 2ACd \\ -ABa + 2A^2 b + (2AC - B^2)c + ABd \end{pmatrix},\]
so that
\[\begin{pmatrix} a \\ b \\ c \\ d \end{pmatrix} = S \begin{pmatrix} a' \\ b' \\ c'  \\ d' \end{pmatrix} = \frac{1}{A|D|} \begin{pmatrix} Aa' - Bb' - Ac' \\ -Cb' - Bc' + Cd' \\ A(b' + d') \\ A(a' + c') \end{pmatrix}.\]
From this, we see that the condition $(a,b,c,d) \in \mathbb{Z}^4$ is equivalent to $(a',b',c',d') \in \mathbb{Z}^4$ satisfying
\[Aa' - Bb' - Ac' \equiv -Cb' - Bc' + Cd' \equiv A(b' + d') \equiv A(a' + c') \equiv 0 \hspace{-.25cm} \pmod{A|D|}.\]
Moreover, we have that
\begin{align*}
Q_1(a,b,c,d) & = \frac{1}{2} \begin{pmatrix} a' & b' & c' & d' \end{pmatrix} A_1' \begin{pmatrix} a' \\ b' \\ c' \\ d' \end{pmatrix} = \frac{1}{A|D|} Q_3(c',d'),	\\
Q_2(a,b,c,d) & = \frac{1}{2} \begin{pmatrix} a' & b' & c' & d' \end{pmatrix} A_2' \begin{pmatrix} a' \\ b' \\ c' \\ d' \end{pmatrix}
= \frac{1}{A|D|^2} \left(Q_3(a',b') - Q_3(c',d')\right),
\end{align*}
so that the conditions $Q_1(a,b,c,d) = n$ and $Q_2(a,b,c,d) = \ell$ are equivalent to $Q_3(c',d') = A|D|n$ and $Q_3(a',b') = A|D|(n + |D|\ell)$.
\end{proof}

\begin{rem}
A consequence of this result is that for $\delta < \frac{1}{|D|\ell}$,
\[\sharp\{\gamma \in \Gamma_{\ell} : u(\gamma z_0,z_0) \leq \delta\} = \sharp\left\{(a',b') \in \mathbb{Z}^2 : {a'}^2 - Ba'b' + AC{b'}^2 = \ell\right\}.\]
This can be seen by noting that if $n = 0$, then necessarily $c' = d' = 0$ in \eqref{eqn:quadformCOV}, so that the congruence constraints on $a'$ and $b'$ are simply $a' \equiv 0 \pmod{|D|}$ and $b' \equiv 0 \pmod{A|D|}$, at which point we may make the change of variables $a' \mapsto |D|a'$ and $b' \mapsto A|D|b'$. This recovers a special case of \cite[Lemma~2]{Mi}.
\end{rem}

\begin{proof}[Proof of Proposition~\ref{prop:Gammaellubound}]
We have that
\begin{multline*}
\sharp\{\gamma \in \Gamma_{\ell} : u(\gamma z_0,z_0) \leq \delta\} \leq \sum_{0 \leq n \leq |D|\ell \delta} \sharp\left\{(c',d') \in \mathbb{Z}^2 : Q_3(c',d') = A|D|n\right\}   \\
\times \sharp\left\{(a',b) \in \mathbb{Z}^2 : Q_3(a',b') = A|D|(n + |D|\ell)\right\}.
\end{multline*}
Since $Q_3$ is positive definite, this in turn is
\[\ll_{D,\epsilon} \sum_{0 \leq n \leq |D|\ell \delta} (n + 1)^{\epsilon} (n + |D|\ell)^{\epsilon},\]
which yields the desired result.
\end{proof}

\begin{rem}
\label{rem:HK}
An alternative yet lengthier approach to proving \eqref{eq:enhan_density} is to use Waldspurger's formula in order to reduce the problem to showing that for each $\xi\in \widehat{\mathrm{Cl}_D}$,
\begin{multline}
\label{eqn:Lfunctiontwistedsecondmoment}
\sum_{\substack{\pi \textrm{ cusp.\ aut.}\\ c(\pi) = 1}} \lambda_{\pi}(\ell) \cdot \frac{L(\frac{1}{2},\pi\otimes \sigma_{\xi})}{L(1,\pi,\mathrm{ad})} e^{-\frac{t_{\pi}^2}{T^2}}   \\
+ \frac{1}{2\pi} \int_{-\infty}^{\infty} \lambda(\ell,t) \cdot \frac{L(\frac{1}{2} + it,\sigma_{\xi}) L(\frac{1}{2} - it,\sigma_{\xi})}{\zeta(1 + 2it)\zeta(1 - 2it)} e^{-\frac{t^2}{T^2}} \, dt \ll_{D,\epsilon} (\ell T)^{\epsilon} \left(\frac{T^2}{\sqrt{\ell}} + \sqrt{\ell}\right).
\end{multline}
For $\ell = 1$, this follows with minor modifications from \cite[Proposition~6.1]{HK}, and the same method remains valid for arbitrary positive integers $\ell$. The proof is via a series of manoeuvres that are relatively standard in the analytic theory of automorphic forms, namely:
\begin{enumerate}
\item Use the approximate functional equation to replace both $L(\frac{1}{2},\pi\otimes\sigma_{\xi})$ and $L(\frac{1}{2} + it,\sigma_{\xi})L(\frac{1}{2} - it,\sigma_{\xi})$ with Dirichlet polynomials, up to a negligibly small error term;
\item Interchange the order of summation, then apply the Kuznetsov formula to express the resulting sum of Fourier coefficients of automorphic forms in terms of a sum of Kloosterman sums;
\item Open up the Kloosterman sums and interchange the order of summation;
\item Apply the Vorono\u{\i} summation formula and interchange the order of summation once more.
\end{enumerate}
At this point, the resulting expression can be identified as a shifted convolution sum involving the Hecke eigenvalues of $\sigma_{\xi}$ weighted by a test function. One can bound this test function using stationary phase techniques and then trivially bound the ensuing shifted convolution sum. This yields the desired bound \eqref{eqn:Lfunctiontwistedsecondmoment}.
\end{rem}


\end{document}